\documentclass[headsepline,11pt,a4paper]{amsart}
\usepackage{amsmath,amssymb,amsfonts,amsthm,exscale,calc}
\usepackage[latin1]{inputenc}
\usepackage{latexsym,textcomp,times}
\usepackage{graphicx}
\usepackage{parskip}
\usepackage{floatflt}
\usepackage{hyperref}
\usepackage{vmargin}
\theoremstyle{theorem}
\newtheorem{thm}{Theorem}
\newtheorem{lemma}{Lemma}
\newtheorem{prop}{Proposition}

\theoremstyle{definition}
\newtheorem*{ex}{Example}

\newcommand{\Hm}[1]{\leavevmode{\marginpar{\tiny%
$\hbox to 0mm{\hspace*{-0.5mm}$\leftarrow$\hss}%
\vcenter{\vrule depth 0.1mm height 0.1mm width \the\marginparwidth}
\hbox to  0mm{\hss$\rightarrow$\hspace*{-0.5mm}}$\\
\relax\raggedright #1}}}

\newcommand{\TT}{\mathbb{T}}
\newcommand{\T}{{\mathbb T}}
\newcommand{\Z}{{\mathbb Z}}
\newcommand{\R}{{\mathbb R}}
\newcommand{\C}{{\mathbb C}}
\newcommand{\h}{{\mathbb H}}
\newcommand{\N}{{\mathbb N}}

\newcommand{\A}{{\mathcal A}}

\newcommand{\Sp}{{\mathbb S}}

\newcommand{\Lp}{{L}}

\newcommand{\clos}{{\mathrm {clos}\,}}

\newcommand{\qand}{{\quad\mathrm {and}\quad}}
\newcommand{\qqand}{{\qquad\mathrm {and}\qquad}}
\newcommand{\dist}{{\mathrm{dist}}}

\newcommand{\al}{{\alpha}}
\newcommand{\de}{{\delta}}
\newcommand{\be}{{\beta}}
\newcommand{\Gm}{{\Gamma}}
\newcommand{\gm}{{\gamma}}
\newcommand{\ph}{{\varphi}}
\newcommand{\lm}{{\lambda}}

\newcommand{\eps}{{\varepsilon}}

\newcommand{\si}{{\sigma}}

\newcommand{\ap}[1]{\left( #1\right)}
\newcommand{\ab}[1]{\left( #1\right)}

\newcommand{\am}[1]{\vert #1\vert_{\arg}}

\newcommand{\na}[1]{\left\vert #1\right\vert_{\arg}}
\newcommand{\ma}[1]{\vert #1\vert_{\arg}}
\newcommand{\ov}[1]{\overline{#1}}

\newcommand{\mo}[1]{\left\vert #1\right\vert}
\newcommand{\ip}[1]{\langle #1\rangle}
\newcommand{\set}[1]{\left\{ #1\right\}}
\let\Im\undefined
\let\Re\undefined
\DeclareMathOperator{\Im}{Im}
\DeclareMathOperator{\Re}{Re}

\begin{document}

\title[spectral theory of trees with finite  cone  type] {On the spectral theory of trees with finite  cone  type}

\author{Matthias Keller}
\author{Daniel Lenz}
\author{Simone Warzel}

\address{Friedrich Schiller Universit\"at Jena, Mathematisches Institut, D-07745 Jena, Germany, m.keller@uni-jena.de.}
\address{Friedrich Schiller Universit\"at Jena, Mathematisches Institut, D-07745 Jena, Germany,\\ daniel.lenz@uni-jena.de,
 URL: http://www.analysis-lenz.uni-jena.de/  }
\address{Technische Universit\"at M\"unchen,
Zentrum Mathematik,
Boltzmannstra\ss e 3,
D-85747 Garching, warzel@ma.tum.de }

\subjclass[2000]{}

\begin{abstract}
We study basic spectral features of graph Laplacians associated with a  class of rooted trees which contains all regular trees. Trees in this class can be generated by substitution processes. Their spectra are shown to be purely absolutely continuous and to consist of finitely many bands. The main result gives stability of absolutely continuous spectrum under sufficiently small radially label symmetric perturbations for non regular trees in this class. In sharp contrast,  the  absolutely continuous spectrum can be completely destroyed by arbitrary small radially label symmetric perturbations for regular trees in this class.
\end{abstract}

\maketitle

\section{Introduction}
The aim of this paper is the investigation of the spectral theory of Laplacians of certain rooted trees.  In these trees the vertices are labeled with finitely many labels each  encoding the whole forward cone of the vertex. Accordingly, these trees can be considered as  trees of  finite cone type. Such trees can  be generated by substitution processes on a finite set. Thus, one can also consider  them  as substitution trees.

Assuming some irreducibility of the underlying substitution, we can completely describe the spectrum of the associated Laplacians: It is purely absolutely continuous and consists of finitely many intervals. In this sense these trees behave like Laplacians on lattices with periodic potentials.

Our main  result then  deals with perturbations by potentials. We assume that the underlying substitution has a strongly connected graph, i.e., each vertex is connected with any other vertex. We then consider potentials which are radially label symmetric  (i.e., within each sphere of vertices the potential depends only on the label of each vertex).  For such potentials with sufficiently small coupling  we show persistence of absolutely continuous spectrum on any   fixed proper piece of the spectrum (away from some finite exceptional energies) provided the tree is non regular. For regular trees on the other hand, the absolutely continuous spectrum can be destroyed by arbitrary small such potentials (as is well known).   In this sense the loss of symmetry (i.e., regularity of the tree) stabilizes the absolutely continuous spectrum.

In a companion work the results of this paper will be used to tackle stability under random perturbations \cite{KLW2}.

Let us put our models and results in perspective. Our basic aim is to study absolutely continuous spectrum and its persistence under (small)  perturbations for certain tree models.

Of course, generically in a topological sense, families of self-adjoint operators  tend to lack an absolutely continuous component in their spectra~\cite{Si} (see also \cite{LS}). More specifically,   in the context of trees,  Breuer \cite{Br} and Breuer/Frank \cite{BF} proved that absolutely continuous spectrum does typically not occur for certain radial tree operators (i.e., in an essentially  one dimensional situation).

Our class of  tree models (and potentials) is  characterized by a certain type of symmetry  condition vaguely reminiscent of a form of periodicity. Thus,  we are in a very non-generic situation and  our stability result does not contradict generic absence of absolutely continuous spectrum.

In the case of vanishing potential, our tree models  have already attracted a lot of attention  in the context  of  random walks under the name of periodic trees or trees with finite cone type \cite{Ly,Tak,NW,Mai} (see \cite{Kro,KT} for related material as well). These works deal with questions such as  recurrence and transience of  random walks. This amounts to studying the limiting behavior of  resolvents at a single energy viz.\ the infimum of the spectrum. Our situation is substantially more complicated as we have to control the limiting behavior for all spectral energies and at the same time  also allow for a potential.


One can also think of our trees as arising as codings of all paths with a fixed initial point in a finite directed graph (the substitution graph). In this sense, our trees are coverings of directed finite graphs.  For abelian covering of finite graphs, absolutely continuous spectrum has recently been shown by Higuchi/Nomura \cite{HN} by means of Fourier/Bloch type analysis. As ours is not an abelian covering such an analysis does not seem to be at our disposal. Still the results  bear a remarkable resemblance.  This is an interesting phenomenon which may be worth further exploration.

Finally,  there is a strong   interest in  absolutely continuous spectrum for trees from the point of view of random Schr\"odinger operators. Such  operators have extensively been studied in the past (see e.g. the monographs \cite{CL,PF,Sto} for details and  further literature). They exhibit localization, i.e., pure point spectrum (under suitable assumptions on the random potential). It  is generally believed that such random operators will also allow for some absolutely continuous spectrum in sufficiently high dimensions. So far, this could not be proven in finite dimensions. However, a remarkable result of Klein \cite{Kl1,Kl2} shows that in the infinite dimensional situation (i.e., for regular trees) one does indeed have persistence of absolutely continuous spectrum. Recent years have seen a new interest in Klein's result. In fact two alternative approaches to his result have been given by Froese/Hasler/Spitzer \cite{FHS1,FHS2} and Aizenman/Sims/Warzel \cite{ASW} respectively. 
Our approach opens the possibility to deal with a substantially larger class of trees. Details will be given in \cite{KLW2}.

The paper is organized as follows. In Section~\ref{s:ModelAndResults} we introduce our model and present the results. Basic properties of the resolvents of the operators are studied in Section~\ref{s:Greenfunction}. In particular the recursion relations of resolvents plays a crucial role in the analysis. These can be translated into a system of polynomial equations, see Section~\ref{s:Polynomial}, and into a fixed point equation of a recursion map, see Section~\ref{s:RecursionMap}.
Finally, using these viewpoints on the recursion relations, the proofs of our main results are given in the last section.


\section{Models and Results}\label{s:ModelAndResults}

A tree is a connected graph without loops.  In a tree with a distinguished vertex called
the root~$o$, the vertices can be ordered according to spheres, i.e., the distance $|\cdot|=d(\cdot,o)$ to the root.

We are interested in special rooted trees generated as follows.
Let $\A$ be a finite set whose elements we call labels and consider a matrix
\begin{align*}
M : \A\times \A\longrightarrow \N_0,\quad(j,k) \longmapsto M_{j,k},
\end{align*}
which we call the substitution matrix. To each label $j\in \mathcal{A}$ we
construct inductively a tree $\TT = \TT (M,j)$ with vertex set $V$ and edges $E$ together with a labeling of the vertices, i.e.,
a function $a : V\longrightarrow \mathcal{A}$. The root of the
tree gets the label $j$.  Each vertex with label $k\in \mathcal{A}$ of the $n$-th sphere is joined by single non-directed edges to $M_{j,k}$ vertices with label $j$ of the $(n+1)$ sphere.

We assume that the tree is not one-dimensional, i.e., if $\A$ consists of one element such that $M$ is thus a natural number, this number is at least  $2$. Furthermore, we impose two conditions on $M$, viz.
\begin{itemize}
  \item [(M1)] $M_{j,j}\geq 1$ for all $j\in\A$ \emph{(positive diagonal)}.
  \item [(M2)] There exists $n=n(M)\in\N$ such that $M^{n}$ has positive entries \emph{(primitivity)}.
\end{itemize}

Let us give two examples of such trees.

\begin{ex} (1.) Assume $\A$ consists of only one element and $M$ is a natural number $k\ge2$. Then, $\T$ is a $k$-regular tree, i.e., a tree where each vertex has exactly $k$ forward neighbors.

(2.) Let $\A=\{\circ,\bullet\}$ and $M=\left(\begin{smallmatrix}
2 & 1 \\
1 & 1 \end{smallmatrix}
\right),
$ where the first row is associated to~$\circ$ and the second one to~ $\bullet$.
Figure~\ref{f:tree} illustrates the tree $\T=\T(M,\bullet)$.
\begin{figure}[!h]
\centering
\scalebox{0.15}{\includegraphics{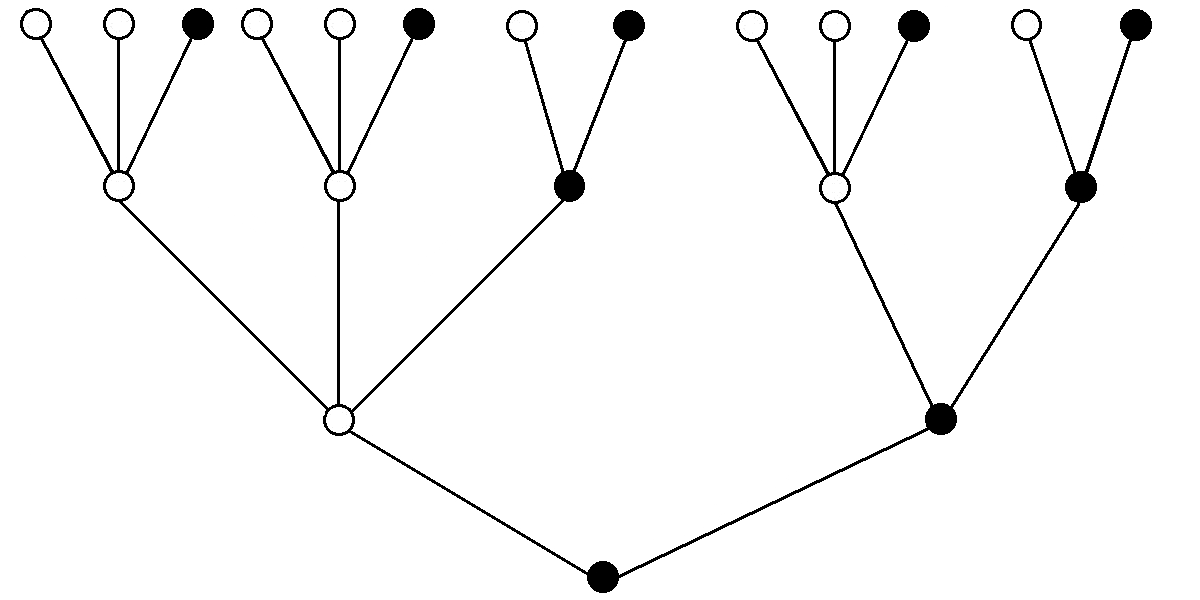}}
\caption{An example of a non-regular tree of finite cone type}\label{f:tree}
\end{figure}
\end{ex}

The topic of this paper are the spectral properties of operators of Laplace-type on $\ell^2(V)$. In order to ease the notation, we present the results and proofs for the case of the adjacency matrix which acts as
\begin{align}\label{e:Lp}
(\Lp \ph)(x):=\sum_{y\sim x}\ph(x),
\end{align}
where $y\sim x$ means that $y$ and $x$ are connected by an edge.
We stress that this choice is only for simplicity and our results also apply in case we include edge weights and/or potentials which are invariant with respect to the labeling (see \cite{Ke} for details).

Our first main  result completely describes  the spectrum of this operator.

\begin{thm} \label{main1}\emph{
There exist finitely many intervals such
that for every $j\in \mathcal{A}$ the spectrum of $\Lp$ associated with the tree $\TT (M,j)$ consists
of exactly these intervals and is purely absolutely continuous.}
\end{thm}

It is well known that  for $\Lp$ on a one-dimensional tree $\T$, i.e., $\Lp$ on $\ell^2(\N)$, the theorem remains true. However, the other assumptions (M1) and (M2) are vital as the following counter example shows.

\begin{ex} A tree constructed from $M$ which does not satisfy $M_{j,j} \geq 1$ may have eigenvalues: Consider for example a tree with two labels and the substitution matrix $M=\left(\begin{smallmatrix}
0 & 2 \\
1 &0  \end{smallmatrix}
\right)$. Then, ${L}$ possesses an eigenfunction corresponding to the eigenvalue $0$. This eigenfunction vanishes on the even spheres and has values $+\frac{1}{2^k}$ and $-\frac{1}{2^k}$ at the vertices of the $2 k +1$ spheres.
\end{ex}

Next, we introduce a class of potentials on $\T$ for which vertices  with the same label and with the same distance $ | \cdot | $ to the root get the same value. More specifically,  we call a  function $v:V\to\R$  \emph{radially label symmetric} if $a(x)=a(y)$ and $|x|=|y|$ implies $v(x)=v(y)$ for all $x,y\in V$.

Our main result is the stability of the absolutely continuous spectrum of the Laplacian under small perturbations by such radially label symmetric potentials in case of non regular trees.

As usual a regular tree is a tree where every vertex has the same number of forward neighbors.

\begin{thm}\label{main2}\emph{
Assume that $\T$ is a non regular tree. Then, for every compact set $I$ contained in the interior of  $\si(\Lp)\setminus\{0\}$ there exist $\lm>0$ such that
for all radially label symmetric $v:V\to[-1,1]$  we have
$$I\subseteq \si_{\mathrm{ac}}(\Lp+\lm v)\qqand I\cap \si_{\mathrm{sing}}(\Lp+\lm v)=\emptyset.$$}
\end{thm}

\textbf{Remarks.}  (a) For regular trees there are potentials $v:V\to[-\lm,\lm]$ which destroy the absolutely continuous spectrum of $\Lp$ completely no matter how small we choose $\lm$.
Examples of such potentials are radially symmetric ones, where the common value of the potential in each sphere is a random variable. Their absolutely continuous spectrum coincides with that of the associated one-dimensional operator and therefore vanishes almost surely \cite{CL,PF}.  As is shown in the above theorem if we exclude regular trees then an arbitrary large part of the absolute continuous spectrum is stable for small radially label symmetric perturbations.

(b) The previous theorem assumes some symmetry of the potential. It may be interesting to study whether this symmetry assumption is indeed necessary for persistence of absolutely continuous spectrum.


(c) For more general operators than $\Lp$, namely those whose edge weights and the diagonal are label invariant, a similar statement holds. There, we have to exclude a finite set of energies instead of only $0$ from $\si(\Lp)$. For details we refer the reader to  \cite{Ke}.

Finally we consider radial label symmetric potential which vanish at infinity.

\begin{thm}\label{main3}\emph{
Assume that $\T$ is a non regular tree. Then, for every radially label symmetric $v$ with $v(x)\to0$ as $|x|\to\infty$ we have
$$\si_{\mathrm{ac}}(\Lp+ v)= \si_{\mathrm{ac}}(\Lp)$$}
\end{thm}

\textbf{Remark.}  Note that we do not need any assumptions on the decay rate of $v$.  We only need that the potential becomes arbitrary small eventually in order to have full preservation of absolutely continuous spectrum. This stands in strong contrast to the one dimensional situation, see e.g. \cite{La} and references therein, or  the situation of regular trees, cf. \cite{D,Br,Ku}.

\section{The recursion formulas of the Green function}\label{s:Greenfunction}
In this section we introduce the Green function of the operator $\Lp$. Moreover, we recall some basic properties such as the recursion formulas for the truncated version of the Green function.

The statements of the first two  subsections are proven for arbitrary rooted trees $\T=(V,E)$, where the operator $\Lp$ acting as \eqref{e:Lp} is bounded on $\ell^2(V)$. This is exactly the case whenever the number of forward neighbors is uniformly bounded in the vertices.  In the last subsection we restrict attention to trees of finite cone type.

\smallskip

For a bounded function $v:V\to\R$ we denote the operator of multiplication also by $v$ and we let
\begin{align*}
    H:=\Lp+v:\ell^{2}(V)\to\ell^2(V),\quad (H\ph)(x)=\sum_{y\sim x}\ph(y)+v(x)\ph(x).
\end{align*}

\subsection{The Green function of the tree}
The spectral measure associated with the characteristic function $\de_x$ of a vertex $x\in V$ is denoted by $\mu_{x}$. Its Borel transform, i.e., the Green function $G_{x}(z,H)$ of $H$ in $x$ is defined on the  upper half plane
$$\h :=\{z\in \C : \Im z>0 \}$$
by
$$G_{x} (z,H):=\int \frac{1}{t-z} d\mu_{x} (t) = \langle \frac{1}{H-z} \delta_x,\delta_x\rangle.$$
Note that the spectral measures $\mu_{x}$ are the vague limits of $\Im G_{x}(E+i\eta,H) dE$ as $\eta\downarrow 0$. 
So, whenever $\Im G_{x}(E+i\eta,H)$ remains bounded for all  energies $E$ in some interval as $\eta\downarrow 0$, the spectral measures $\mu_{x}$ is purely absolutely continuous.

An effective  way for its computation is to look at the Green function of truncated trees at some vertex. We denote by $\T_x$ the forward tree of a vertex $x\in V$ with respect to the root of $\T$, i.e., if one deletes the unique edge that is adjacent to $x$ in the path connecting $x$ to the root of $\T$, then $\T_{x}$ is the connected component that contains $x\in V$. We define the truncated resolvents or Green functions by
\begin{align*}
    \Gm_{x} (z,H)= \langle \frac{1}{H_{\T_{x}}-z} \delta_x,\delta_x\rangle,
\end{align*}
where $H_{\T_{x}}$ denotes the restriction of $H$ to $\ell^2(\T_{x})$. It can easily  be checked that $G_{x}(\cdot,H)$ and $\Gm_{x}(\cdot,H)$ map $\h$ to $\h$. Moreover, $\Gm_{x}(z,H)=G_{x}(z,H_{\T_{x}})$ for $x\in V$ and thus $\Gm_{o}(z,H)=G_{o}(z,H)$ for the root vertex $o$ of $\T$.

\subsection{Recursion formulas}

The truncated resolvents $\Gm_{x}(z,H)$, $x\in V$, obey a recursion relation. This will be discussed next and is already found in a similar form in \cite{Kl1,ASW,FHS1,FHS2}. For a tree $\T$ with root $o$ and $x\in V$ we define
\begin{align*}
    S_{x}:=\{y\in V\mid y\sim x,\; |y|=|x|+1\}.
\end{align*}

\begin{prop}\label{p:Gm}(Recursion formula I) Let $\T$ be a rooted tree. Then, for $z\in\h$, $x\in V$,
\begin{align}\label{e:Gm}
-\frac{1}{\Gm_{x}(z,H)}=z-v(x)+\sum_{y\in S_{x}}\Gm_{y}(z,H).
\end{align}
\begin{proof}Let $\Lambda$ be the self-adjoint operator which connects $x\in V$ to its forward neighbors, i.e., $\ip{\Lambda\de_x,\de_y}=\ov{\ip{\Lambda\de_y,\de_x}}=1$ for $y\in S_{x}$ and all other matrix elements vanish. Then, $H':=H-\Lambda$ is a direct sum of operators and the resolvent identity applied twice yields
\begin{align*}
\frac{1}{H-z}=\frac{1}{H'-z} -\frac{1}{H'-z}\Lambda\frac{1}{H'-z} +\frac{1}{H'-z}\Lambda\frac{1}{H'-z}\Lambda\frac{1}{H-z}.
\end{align*}
Taking matrix elements we conclude \eqref{e:Gm}.
\end{proof}
\end{prop}

A variant of this reasoning allows one to relate $G_{x}(z,H)$ and $\Gm_{x}(z,H)$.\medskip

\begin{prop}\label{p:G}(Recursion formula II) Let $\T$ be a rooted tree. Then, for $z\in\h$, $x\in V$ and $y\in S_{x}$
\begin{align*}
G_{y}(z,H)={\Gm_{y}(z,H)}+\Gm_{y}(z,H)^{2}G_{x}(z,H).
\end{align*}
\begin{proof}Applying the resolvent identity twice, one sees
\begin{align*}
\frac{1}{H-z}=\frac{1}{H'-z} -\frac{1}{H'-z}\Lambda\frac{1}{H'-z} +\frac{1}{H'-z}\Lambda\frac{1}{H-z}\Lambda\frac{1}{H'-z}
\end{align*}
and the statement follows.
\end{proof}
\end{prop}

We can draw the following consequences from these formulas.

\begin{prop}\label{p:G2} Let $E\in\R$.

$\mathrm{(1.)}$ If $\Gm_{x}(E+i\eta, H)$ is uniformly bounded in $\eta>0$ for all $x\in V$, then $G_{x}(E+i\eta,H)$ is uniformly bounded in $\eta>0$ for all $x\in V$.

$\mathrm{(2.)}$ If $\Gm_{x}(E+i\eta,H)$ is uniformly bounded in $\eta>0$ and  $\lim_{\eta\downarrow0}\Im\Gm_{x}(E+i\eta,H)=0$ for all $x\in V$, then     $\lim_{\eta\downarrow0}\Im G_{x}(E+i\eta,H)=0$ for all $x\in V$.

$\mathrm{(3.)}$ If $\Gm_{x}(E,H):=\lim_{\eta\downarrow0}\Gm_{x}(E+i\eta,H)$ exists and $\Im \Gm_{x}(E,H)>0$  for all $x\in V$, then $G_{x}(E,H):=\lim_{\eta\downarrow0}G_{x}(E+i\eta,H)$ exists and $\Im G_{x}(E,H)>0$ for all $x\in V$.
\begin{proof}
(1.), (2.) and the statement about the existence of the limits in (3.) directly  follow from Proposition~\ref{p:G} by induction on the distance to the root.

It remains to show the statement in (3.) about positivity of the imaginary parts. For $x\in V$ let $x_0\sim x$ be the unique vertex on the path connecting $x$ with the root $o$. Applying the recursion relation \eqref{e:Gm} to the tree in which one singles out $x$ as a root, one obtains
\begin{align*}
-\frac{1}{G_{x}(z,H)}=z-v(x)+\sum_{y\in S_x}
\Gm_{y}(z,H) +G_{x_0}(z,H_{\T\setminus \T_{x}}),
\end{align*}
for $z=E+i\eta\in\h$.
We estimate $\Im G_{x_0}(z,H_{\T\setminus \T_{x}})>0$, go to the limit $\eta\downarrow0$, take imaginary parts and multiply by $|G_{x}(E,H)|^{2}$ to get
\begin{align*}
\Im G_{x}(E,H) \geq
|G_{x}(E,H)|^{2}{\sum_{y\in S_x}\Im\Gm_{y}(E,H)}.
\end{align*}
To conclude positivity of the left hand side, we have to show $|G_{x}(E,H)|>0$. To see this we apply the recursion formula \eqref{e:Gm} to $G_{x_0}(z,H_{\T\setminus \T_{x}})$ (with respect to the  tree $\T$ with root $x_0$) in the first equation of the proof. We take the modulus, take the limit $\eta\downarrow0$  and  obtain
\begin{align*}
\frac{1}{|G_{x}(E,H)|}\leq\mo{E-v(x)}+\sum_{y\in S_x}|\Gm_{y}(E,H)| +\frac{1}{\Im\Gm_{y_0}(E,H)},
\end{align*}
where we estimated the denominator of the last term first by its imaginary part and then dropped all but one term for some vertex  $y_0 \neq x $ neighboring $ x_0 $. Since we assumed $\Im \Gm_{y}(E,H)>0$ for all $y\in V$ the statement follows.
\end{proof}
\end{prop}

\subsection{Recursions for trees of finite cone type}
 Here, we  turn to the situation that a substitution matrix $M:\A\times\A\to\N_0$ on a finite label set $\A$ and a corresponding tree $\T=\T(M,j)$, $j\in\A$, are given. Then, the truncated Green function $\Gm_{x}(z,\Lp)$, $x\in V$, depends only on the label $a(x)$ of $x$. We define $\Gm(z)\in \h^{\A}$ via
\begin{align}\label{e:Gmdef}
\Gm_{a(x)}(z):=\Gm_{x}(z,\Lp).
\end{align}
The recursion relations \eqref{e:Gm}  consequently  reduce to finitely many equations:
\begin{align}\label{e:Gm2}
-\frac{1}{\Gm_{j}(z)}=z+\sum_{k\in\A}M_{j,k}\Gm_{k}(z),\quad j\in\A.
\end{align}

\section{Polynomial equations}\label{s:Polynomial}

Our setting in this section is the following: Let a substitution matrix $M:\A\times\A\to\N_0$ on a finite label set $\A$ be given. Suppose further that $M$ has positive diagonal (M1) and is primitive (M2). Let $\ov \h$ be the closure of $\h$ in $\C$ i.e.
\begin{align*}
   \ov \h=\h\cup\R =\{ z \in \C: \Im z \geq 0\}.
\end{align*}
We consider the polynomials
\begin{equation*}
P_j:\C\times\C^{\A}\to\C,\quad(z,\xi)\mapsto {z\xi_j+\sum_{k\in\A}M_{j,k}\xi_k\xi_j}+1, \quad j\in\A,
\end{equation*}
and look for solutions $\xi \in \h^{\A}$
\begin{align}\label{e:Polynom}
P_{j}(z,\xi)=0,\quad j\in\A,
\end{align}
for fixed $z\in\ov\h$. This is relevant as $\Gm(z)$, (see \eqref{e:Gmdef} for definition,) is in $\h^{\A}$ for $z\in \h$ and solves
$P_{j}(z,\Gm(z))=0$, $j\in\A$ by \eqref{e:Gm2}.

In order to study $\Gm(z)$ in the limit $\Im z\downarrow0$, we will be concerned with the subset of $\R$ where the polynomial equations have a solution in $\h^{\A}$. To this end, we will introduce in the following two sets $\Sigma_0$ and $\Sigma_1$ and study their properties.

\subsection{The set $\Sigma_1$ and an application of a theorem of Milnor}
Set
\begin{align*}
    \Sigma_1:=\{E\in\R\mid P_{j}(E,\xi)=0 \mbox{ for some } \xi\in\h^{\A}\},
\end{align*}
whose closure turns out to be the spectrum of $\Lp$.

\begin{lemma}\label{l:milnor}The set $\Sigma_1$ consists of finitely many intervals.
\begin{proof}
Set
\begin{align*}
    Z:=\{(E,u,v)\in\R\times\R^{\A}\times\R^{\A}\mid P_j(E,u+iv)=0\mbox{ and $v_{j}>0$ for all $j\in\A$}\}.
\end{align*}
As $Z$ is given by a system of polynomial (in)equalities, it has finitely many components by a theorem of Milnor~\cite{Mil}. (The result deals with inequalities of the form $\geq$ but this can easily be carried over to strict inequalities, see also Lemma 3.5 \cite{Ke}). As
\begin{align*}
\Sigma_1=\mathrm{Pr}_1 Z,\quad\mbox{with}\quad
\mathrm{Pr}_1:\R\times\R^{\A}\times\R^{\A}\to\R,\quad(E,u,v)\mapsto E,
\end{align*}
and $\mathrm{Pr}_1$ is continuous, the set $\Sigma_1$ has finitely many connected components as well.
\end{proof}
\end{lemma}

\subsection{The set $\Sigma_0$ and the case of non regular trees}\label{ss:Sigma0}
We consider the subset of $\Sigma_1$ where the components of a solution are all linear multiples of each other, i.e., let
\begin{align*}
\Sigma_0:=\{E\in\R\mid P_{j}(E,\xi)=0 \mbox{ and } \arg \xi_{j}\ov\xi_{k}=0\mbox{ for some } \xi\in\h^{\A}\mbox{ and all } j,k\in\A\},
\end{align*}
where $\arg$ is the argument of a non zero complex number. We consider $\arg$ as a map
$$ \arg : \C\setminus\{0\}\longrightarrow \Sp^{1}\cong\R/2\pi\Z$$  and note that it is a continuous group homomorphism.

Clearly, in the case of regular trees, where $\A$ is a singleton set and $M$ a positive integer, we have $\Sigma_0=\Sigma_1$. However, we will show that whenever $M$ induces a non regular tree, $0$ is the only possible element of $\Sigma_0$. This means for all other energies there is a solution which has two components such that the argument of the product is non zero. Therefore, they share a non vanishing angle.

Let $\T=\T(M,j)$, $j\in\A$, be a tree associated with the substitution matrix $M$. Note that $\T$ is a regular tree if and only if there is a number $k\in\N$ such that $k=\sum_{l\in\A}M_{i,l}$ for all $i\in\A$.

\begin{lemma}\label{l:Sigma0}
Suppose $\T$ is not a regular tree. Then, $\Sigma_0\subseteq\{0\}$.
\begin{proof}
In this proof  we denote $\A=\{1,\ldots,N\}$.
Let $E\in\Sigma_1$. Then there is $\xi\in\h^{\A}$ such that $P(E,\xi)=0$. We assume now that $\arg \xi_j\ov{\xi_k}=0$ for all $j,k\in\A$ which is equivalent to assuming that there are $r_{j}>0$ such that

\begin{align*}
    \xi_{j}=r_{j}\xi_1, \qquad j\in\A.
\end{align*}
In this case $E\in \Sigma_0$ by definition of $\Sigma_0$. With this assumption the polynomial equations \eqref{e:Polynom} become
\begin{align*}
0=\sum_{k\in\A}M_{j,k}r_{j}r_{k}\xi_1^{2}+Er_j\xi_{1}+1.
\end{align*}
We denote the coefficients of $\xi_1$ by
\begin{align*}
    c_{2,j}=\sum_{k\in\A}M_{j,k}r_{j}r_{k}\qand c_{1,j}=Er_j, \qquad j\in\A.
\end{align*}
Quadratic polynomials with real coefficients have either two complex conjugated roots or two real ones. Since we assumed $\xi\in\h^{\A}$, we already know one root is $\xi_1\in\h$, so the other one must be $\ov{\xi_{1}}$. The polynomials all have the same roots and therefore  must have the same coefficients, because they are normalized. We consider two cases.

\emph{Case 1: $c_{1,j}=0$ for some (all) $j\in\A$.} We immediately see that this can only happen if $E=0$.

\emph{Case 2: $c_{1,j}\neq0$ for all $j\in\A$.}
By the assumption $c_{1,1}=\ldots=c_{1,N}$ we  obtain
$ r_{1}=\ldots=r_{N}=1$ as also $r_1=1$. Hence, we get that $c_{2,1}=\ldots=c_{2,N}$ if and only if the tree is regular. This case is excluded by assumption.

We conclude that $E=0$ is the only possible element of $\Sigma_0$.
\end{proof}
\end{lemma}

\textbf{Remark.} If one considers operators with label symmetric weights and potentials one can show a similar result. In particular, one finds that $\Sigma_0$ consist of at most $|\A|-1$ energies. For details see \cite[Lemma~3.4]{Ke}.

\section{Recursion maps}\label{s:RecursionMap}

Again,  we consider a substitution matrix $M:\A\times\A\to\N_0$ on a finite label set $\A$ which has positive diagonal (M1) and is primitive (M2).

For given $\zeta\in \ov\h^{\A}$, we define the recursion map     $\Phi_{\zeta}:\h^{\A}\to\h^{\A}$ via
\begin{align*}
\Phi_{\zeta,j}(g)=-\frac{1}{\zeta_j+\sum_{k\in\A} M_{j,k}g_{k}},\qquad j\in\A.
\end{align*}
For $z\in\ov \h$ and $\zeta=(z,\ldots,z)$, we write, in slight abuse of notation, $\Phi_{z}$ for $\Phi_{\zeta}$.

From \eqref{e:Gm2}, it can be seen that $\Phi_{z}(\Gm(z))=\Gm(z)$. In other words, $\Gm(z)$ is a fixed point of $\Phi_{z}$.

Let $v$ be a radially label symmetric potential. By its symmetry, $\Gm_{x}(z,\Lp+v)$, $z\in \h$, depends only on $|x|$ and $a(x)$. Letting  $\zeta_{j}=z+v(x)$ for some $x\in V$ with $a(x)=j$ we have that $\Phi_{\zeta,j}(\Gm'(z,\Lp+v))=\Gm_{x}(z,\Lp+v)$, where $\Gm'_{k}(z,\Lp+v)=\Gm_{y}(z,\Lp+v)$ with $y\in V$ such that $|y|=|x|+1$ and $a(y)=k$.

The aim of this section is to show that suitable 'powers' of the  $\Phi_{\zeta}$ form a  contraction. From that we conclude uniqueness and continuity of fixed points and derive the corresponding properties for the truncated Green functions.

\subsection{Uniform bounds of fixed points}

The structure of the recursion map $\Phi_{z}$ immediately implies  certain upper and lower bounds for its fixed points and thus for $\Gm(z)$.

\begin{lemma}\label{l:Gmbounds} (Uniform bounds) Let $z\in \ov\h$, $h\in\h^{\A}$ be a fixed point of $\Phi_{z}$. Then, for all $j\in\A$,
$$|h_j|^2  \sum_k M_{j,k} \Im h_k \leq \Im h_j$$
 and,  in particular,
\begin{align*}
\frac{1}{\mo{z}+\sum_{k\in\A}M_{j,k}/\sqrt{M_{j,j}}}\leq|h_{j}| \leq \frac{1}{\sqrt{M_{j,j}}}.
\end{align*}
\begin{proof}
Taking imaginary parts in the equality $h=\Phi_{z}(h)$, we get
\begin{align*}
    \Im h_{j}=\frac{{\Im z +\sum_{k\in\A}M_{j,k}\Im h_{k}}}{\mo{z +\sum_{k\in\A}M_{j,k} h_{k}}^2}=\ap{\Im z +\sum_{k\in\A}M_{j,k}\Im h_{k}}\mo{h_{j}}^2.
\end{align*}
Dropping the positive terms $\Im z$ and $M_{j,k}\Im h_{k}$ for $k\neq j$ on the right hand side yields the first statement as well as
\begin{align*}
\Im h_{j}\geq M_{j,j}\Im h_{j}|h_{j}|^{2}.
\end{align*}
As $M_{j,j}\geq1$ and $\Im h_{j}>0$ by assumption, we obtain the upper bound. Given the upper bound, the lower bound follows immediately by taking the modulus in the equation $h=\Phi_{z}(h)$.
\end{proof}
\end{lemma}

\textbf{Remark.} (a) As the truncated Green function is a fixed point, the previous lemma implies that the spectral measure $\mu_o$ with  respect to the characteristic function of the root $o$ of $\T(M,j)$, $j\in\A$, is purely absolutely continuous. Using  Proposition~\ref{p:G2}, we  can  derive that the  spectrum of $\Lp$ is purely absolutely continuous.

(b) For the proof of the lemma  we only used that $M$ has positive diagonal, i.e., assumption (M1). So, whenever there is one label $j\in\A$ such that $M_{j,j}\geq1$ we can infer the existence of absolutely continuous spectrum.

The uniform bounds have an immediate implication, namely that the limits of fixed points are either in  $\h^{\A}$ or $\R^{\A}$.

\begin{lemma}\label{l:accumulationpoints}
For given $E\in\R$ let $A_{E}\in\ov\h^{\A}$ be the  set of accumulation points for sequences $(h_n)$ of fixed points  of $\Phi_{z_n}$, where $z_n \in \ov\h$ such that $z_{n}\to E$. Then, $A_{E}$ is not empty and every $h\in A_E$ is a fixed point of $\Phi_{E}$ and is either in $\h^{\A}$ or in $\R^{\A}$.
\begin{proof} As the truncated Green functions are fixed points of $\Phi_{z}$ for $z\in\h$, the non-emptiness of $A_E$ is implied by the uniform bounds of the previous lemma. The fixed point property of accumulation points in $A_E$ follows from the continuity of $\Phi_{z}$.\\
Let $h\in A_E$ and assume that $\Im h_{j}>0$ for some $j\in\A$. Let $k\in\A$ with $M_{k,j}\geq1$. It follows $\Im h_{k}>0$ since  $\Im h_{k}\geq M_{k,j}\Im h_{j}|h_{k}|^{2}$  by  the lemma above and $|h_{k}|^{2}>0$ by the lower bound of the previous lemma. As $M$ is primitive, (M2), this argument can be iterated for all $k\in\A$. Therefore $\Im h_{k}>0$ for all $k\in\A$ and hence $h\in\h^{\A}$.
\end{proof}
\end{lemma}

\subsection{A decomposition and basic contraction properties}

In order to study the contraction properties of $\Phi_{\zeta}$, we introduce the following metric on $\h^{\A}$
$$\dist_{\h^\A}(g,h)=\cosh^{-1}\ab{\frac{1}{2}\gm_\A(g,h)+1},$$
where $\gm_\A:\h^\A\times\h^\A\to[0,\infty)$ is given by
\begin{align*}
\gm_\A(g,h):=\max_{j\in \A}\gm(g_j,h_j)\qand
\gm(g_{j},h_{j}):=\frac{|g_{j}-h_{j}|^2}{\Im g_{j}\Im h_{j}}
\end{align*}
This approach was also taken by \cite{FHS1} and in \cite{FHS2} a similar function denoted by $cd$ was introduced. Since $\dist_{\h^{\A}}=\max_{j\in \A}\dist_{\h}$ with the usual hyperbolic metric  $\dist_{\h} = \cosh^{-1} \gamma $  of the upper half plane, see \cite[Theorem~1.2.6]{Ka}, $\dist_{\h^{\A}}$ is a metric. For $h\in\h^{^\A}$ and $R\geq0$ define
\begin{align*}
    B_{R}(h):=\{g\in\h^{\A}\mid \gm_{\A}(g,h)\geq R\}.
\end{align*}
We can decompose the map $\Phi_{\zeta}=\rho\circ\sigma_{\zeta}\circ\tau$ into $\rho$, $\sigma_{\zeta}$ and $\tau$ which are given by
\begin{align*}
    \rho_{j}(g)=-\frac{1}{g_{j}}, \quad
\si_{\zeta,j}(g)=\zeta_{j}+g_{j},\quad
\tau_{j}(g)=\sum_{k\in\A}M_{j,k}g_{k},
\end{align*}
for $j\in\A$. Their properties are summarized in Lemma~\ref{l:contraction} below. It ensures that $\Phi_{\zeta}$ is a \emph{quasicontraction} for any $\zeta\in \ov\h^{\A}$, i.e., $\dist_{\h^{\A}}(\Phi_{\zeta}(g),\Phi_{\zeta}(h)) \leq\dist_{\h^{\A}}(g,h)$ for all $g,h\in\h^{\A}$. In the case of strict inequality $\Phi_{\zeta}$ would be called a \emph{contraction}.

\begin{lemma}\label{l:contraction}(Basic contraction properties of $\Phi_{\zeta}$) Consider the metric space $(\h^{\A},\dist_{\h^{\A}})$.
\begin{itemize}
\item [(1.)] $\rho:\h^{\A}\to\h^{\A}$ is an isometry.
\item [(2.)] $\si_{\zeta}:\h^{\A}\to\h^{\A}$ is a quasicontraction for all $\zeta\in \ov\h^{\A}$. If $\Im \zeta_j=0$ for all $j\in\A$ it is an isometry and if $\Im \zeta_j>0$ for all $j\in\A$ it is a contraction which is uniform on compact sets.
\item [(3.)] $\tau:\h^{\A}\to\h^{\A}$ is a quasicontraction and for every $j\in\A$
\begin{align*}
\gm(\tau_j(g),\tau_j(h))={\sum_{k\in\A} p_{j,k}(h)\ab{\sum_{l\in\A}p_{j,l}(g)P_{k,l}(g,h)\cos\al_{k,l}(g,h)}\gm(g_k,h_k)},
\end{align*}
where for $g,h\in\h^{\A}$ and $j,k,l\in\A$
\begin{align*}
p_{j,k}(h)&=M_{j,k}\frac{\Im h_{k}}{\Im\tau_{j}(h)},\quad p_{j,l}(g)=M_{j,l}\frac{\Im g_{l}}{\Im\tau_{j}(g)},\\
P_{k,l}(g,h)&=
\frac{\ab{\Im g_{k}\Im h_{k}\gm(g_{k},h_{k})\Im g_{l}\Im{h_{l}}\gm(g_{l},h_{l})}^{\frac{1}{2}}} {\frac{1}{2}\ab{\Im g_{k}\Im h_{l}\gm(g_{l},h_{l})+\Im g_{l}\Im h_{k}\gm(g_{k},h_{k})}},\\
\al_{k,l}(g,h)&=\arg(g_{k}-h_{k})\ov{\ab{g_{l}-h_{l}}},
\end{align*}
whenever $g_{k}\neq h_{k}$, $g_{l}\neq h_{l}$ and otherwise the corresponding term in the sum are zero.
\end{itemize}
\begin{proof} (1.) For $g,h\in\h^{\A}$ and $j\in \A$ we have
\begin{align*}
\gm(\rho_{j}(g),\rho_{j}(h))=\frac{|g_{j}|^{-2}|h_{j}|^{-2} |g_{j}-h_{j}|^{2}}{\Im g_{j}|g_{j}|^{-2}\Im h_{j}|h_{j}|^{-2}}=\gm(g_{j},h_{j})
\end{align*}
and hence $\dist_{\h^{\A}}(\rho(g),\rho(h))=\dist_{\h^{\A}}(g,h)$.\\
(2.) A direct  calculation yields
\begin{align*}
\gm(\sigma_{\zeta,j}(g),\sigma_{\zeta,j}(h)) &=\frac{|g_{j}-h_{j}|^{2}} {\Im(g_{j}+\zeta_{j})\Im(h_{j}+\zeta_{j})} =\frac{\gm(g_{j},h_{j})}{\Im (1+\zeta_{j}/g_{j})\Im(1+\zeta_{j}/h_{j})}&\leq \gm(g_{j},h_{j}).
\end{align*}
The inequality is strict if $\Im \zeta_{j}>0$. Strict monotonicity of $\cosh^{-1}$ implies the first part of the claim. Finally on compact sets the uniformity of the contraction follows from Lemma~\ref{l:dist} which is stated and proven below.\\
(3.) To prove the third part, we observe that
\begin{align*}
\mo{\tau_{k}(g)-\tau_{k}(h)}^{2}
&=\mo{\sum_{k\in\A}M_{j,k}(g_{k}-h_{k})}^{2} =\sum_{k,l\in\A}M_{j,k}M_{j,l}\cos\al_{k,l} |g_{k}-h_{k}||g_{l}-h_{l}|\\
&=\sum_{k,l\in\A}M_{j,k}M_{j,l}\cos\al_{k,l}\ab{\Im g_{k}\Im h_{k}\Im g_{l}\Im h_{l}\gm(g_{k},h_{k})\gm(g_{l},h_{l})}^{\frac{1}{2}}\\
\end{align*}
and hence the formula for $\tau$ follows. Note that $P_{k,l}\in[0,1]$, since it is a quotient of a geometric and an arithmetic mean. Moreover, $\cos\al_{k,l}\in[-1,1]$ and $\sum_{k,l}p_{j,k,l}=1$. Hence, $\tau$ is a quasicontraction.
\end{proof}
\end{lemma}

The previous proof relied on the following lemma.
\begin{lemma}\label{l:dist}Let $K\subseteq \h^{\A}$ be compact,  $f:\h^{\A}\to\h^{\A}$ and $c_0<1$ such that $\gm_{\A}(f(g),f(h))\leq c_0\gm_{\A}(g,h)$ for all $g,h\in K$. Then there is $c_1<1$ such that $\dist_{\h^{\A}}(f(g),f(h))\leq c_1\dist_{\h^{\A}}(g,h)$ for all $g,h\in K$.
\begin{proof}
We prove that the function $c(r):=\cosh^{-1}(c_{0}r+1)/\cosh^{-1}(r+1)$ is strictly smaller than one. This follows by monotonicity of $\cosh^{-1}$ and $c(0)=\sqrt{c_{0}}$ (which can be checked via L'Hospitals theorem). As $c(r)\geq \dist_{\h^{\A}}(f(g),f(h))/\dist_{\h^{\A}}(g,h)$ for $r=\gm(g,h)/2$, the statement of the lemma follows.
\end{proof}
\end{lemma}

The following statement follows from the basic contraction properties of $\Phi_{\zeta}$. As it is a variant of \cite[Theorem 3.6]{FHS1} (see \cite[Theorem 2.20]{Ke} as well)  we only give a sketch of the proof here.

\begin{prop}\label{p:FHS} Let $z\in\h$, $v$ be a radially label symmetric potential and $\zeta(n)\in\h^{\A}$ with $\zeta_{a(x)}(n)=z+v(x)$ for $n\in\N_0$, $x\in S^{n}$. Then, for all $g\in\h^{\A}$
\begin{align*}
\Gm_{x}(z,\Lp+v) =\lim_{n\to\infty} \Phi_{\zeta({|x|}),a(x)}\circ\Phi_{\zeta({|x|+1})} \circ\ldots\circ\Phi_{\zeta(n)}(g),\quad  x\in V,
\end{align*}
In particular, for $v\equiv0$, the function $\Gm(z,\Lp)=(\Gm_{j}(z,\Lp)_{j\in\A}$ is the unique solution to the polynomial equations \eqref{e:Polynom} and the unique fixed point of $\Phi_{z}$.
\begin{proof}[Sketch of the proof] One checks by direct calculation (see \cite{FHS1,Ke}) that for each $n\in \N_0$  the composition $\Phi_{\zeta_n}\circ\Phi_{\zeta({n+1})}$   maps $\h^{\A}$ into a compact set. According to Lemma~\ref{l:contraction}, the functions $\Phi_{\zeta}$ are uniform contractions on compact sets for $\zeta\in\h^{\A}$. This easily gives the existence of the limit. As $\Gm_{x}(z,\Lp+v)$, $x\in V$, is a solution to the recursion relation \eqref{e:Gm}, it must be equal to the limit. The statements for $v\equiv0$ follow since the recursion relation \eqref{e:Gm2} can be directly translated into the system of polynomial equations \eqref{e:Polynom} and a fixed point equation of the recursion map.
\end{proof}
\end{prop}

The proposition above used positivity of the imaginary parts of $\zeta$ in order to conclude that $\Phi_{\zeta}$ is a uniform contraction. For the study of the spectrum of our operators we have to consider the limit of the imaginary parts tending to zero. A look at the statements of Lemma~\ref{l:contraction} tells us that in this case uniform contraction can only come from $\tau$. This is investigated in the next lemma.

We let  $d_{\Sp^1}:\Sp^{1}\times\Sp^{1}\to[0,2\pi)$ be the canonical translation invariant metric of the group $\Sp^{1}$ and define
\begin{align*}
\ma{\al}:=d_{\Sp^{1}(\al,1)},\quad\al\in\Sp^{1}.
\end{align*}
 Then, $\ma{\cdot}$ satisfies the triangle inequality, i.e., $\ma{\al+\be}\leq\ma{\al}+\ma{\be}$ for $\al,\beta\in\Sp^{1}$.

\begin{lemma}\label{l:tau} (Sufficient criterion for uniform contraction.) Let  $K\subset \h^{\A}$ be relatively compact. Suppose there is $h\in K$ such that
\begin{align*}
    \min_{g\in K'}\max_{j,k\in\A} \ma{\al_{j,k}(g,h)}>0,
\end{align*}
where $K':=\{g\in K\mid g_{j}\neq h_{j}\mbox{ for all }j\in\A\}$.
Then there is $\de>0$ such that for all $g\in K$ and $\zeta(1),\ldots,\zeta(n)\in\ov\h^{\A}$
\begin{align*}
\gm_{\A}\ap{\Phi_\zeta^{(n)}(g),\Phi_\zeta^{(n)}(h)}\leq (1-\de)\gm_{\A}(g,h),
\end{align*}
where  $\Phi_{\zeta}^{(n)}:=\Phi_{\zeta(1)}\circ \cdot \circ \Phi_{\zeta(n)}$ and $n=n(M)$ is the primitivity exponent of $M$ from (M2).
\begin{proof} Recall that the argument $\arg$ is a continuous group homomorphism $\C\setminus\{0\}\to\Sp^{1}$.

We start with a few claims.

\emph{Claim~1: For $g\in K\setminus K'$ there are $k,l\in\A$ such that $M_{k,l}\geq1$ and $P_{k,l}(g,h)=0$.}\\
Proof of Claim~1. If $g\in K\setminus K'$ there is $k\in\A$ such that $g_{k}=h_{k}$ which readily gives $P_{k,l}=0$ for all $l\in\A$ by definition.

\emph{Claim~2: There is $\de'>0$ such that for every $g\in K'$ there are $k,l\in\A$ such that $M_{k,l}\geq1$ and $\ma{\al_{k,l}(g,h)}\geq\de'$.}\\
Proof of Claim~2. By assumption there is $\eps>0$ such that for all $g\in K'$ and suitable $i,j\in\A$ (depending on $g$) we have $\ma{\al_{i,j}(g,h)}\geq\eps$.
By the primitivity assumption (M2) there are $l(1),\ldots, l({n+1})\in\A$ with ${l(1)}=i$, ${l(n+1)}=j$ and $M_{l({s}),l({s+1})}\geq1$ for all $s=1,\ldots, n$. We calculate, using the definition of $\al_{j,k}$ and the triangle inequality,
\begin{align*}
\eps\leq \ma{\al_{i,j}(g,h)}= \na{\sum_{s=1}^{n}\al_{l(s),l({s+1})}(g,h)} \leq \sum_{s=1}^{n}\na{\al_{l(s),l({s+1})}(g,h)}
\end{align*}
and infer the claim by letting $\de'=\eps/n$.

\emph{Claim~3: There is $\de''>0$ such that for every $g\in K$ there exists $k\in\A$ such that}
\begin{align*}
\gm(\tau_k(g),\tau_k(h))
&\leq (1-\de'')\gm_{\A}(g,h).
\end{align*}
Proof of Claim~3. Let $\de'>0$ be taken from Claim~2 and set
\begin{align*}
\de_0:=1-\cos {\de'},\qquad
\de_{1}:=\min_{g\in K}\min_{j,k\in\A}\frac{\Im g_{k}}{\Im\tau_{j}(g)}.
\end{align*}
As $K$ is relatively compact, $\de_1>0$ .
Moreover, for $g,h\in K$ and $j,k\in\A$, let
\begin{align*}
c_{j,k}(g,h)&:={\sum_{l\in\A}p_{j,l}(g) P_{k,l}(g,h)\cos\al_{k,l}(g,h)}.
\end{align*}
Note that $-1 \leq c_{j,k}\leq 1$. For given $g\in K$, we let $k,l\in\A$ be taken from Claim~1 if $g\in K\setminus K'$ and from Claim~2 if $g\in K'$. Hence, $P_{k,l}\cos\al_{k,l}\leq 1-\de_0$.
This, combined with the facts $P_{k,i}\cos\al_{k,i}\leq1$ for $i\neq l$, $\sum_{j}p_{k,j}=1$ and $M_{k,l}\de_0\leq p_{k,l}(h)$, yields
\begin{align*}
c_{k,k}(g,h)&\leq\sum_{i\neq l}p_{k,i}(g)+p_{k,l}(g)P_{k,l}\cos\al_{k,l}
= 1-p_{k,l}(g)(1-P_{k,l}\cos\al_{k,l})\leq 1-M_{k,l}\de_1\de_0.
\end{align*}
Invoking Lemma~\ref{l:contraction}~(3.), using $\gm(g_l,h_l)\leq\gm_{\A}(g,h)$, estimating $c_{k,l}\leq1$ for $l\neq k$, $\sum_{l}p_{k,l}=1$, $M_{k,k}\de_0\leq p_{k,k}(h)$ and employing  the estimate above, we compute
\begin{align*}
\gm(\tau_k(g),\tau_k(h))&=\ap{ \sum_{l\in\A}p_{k,l}(h)c_{k,l}(g,h) \frac{\gm(g_{l},h_{l})}{\gm_{\A}(g,h)}}\gm_{\A}(g,h)\\&\leq\ap{ \sum_{l\in\A, l\neq k}p_{k,l}(h)+p_{k,k}(h)c_{k,k}(g,h)}\gm_{\A}(g,h)
\\&\leq
(1-p_{k,k}(h)(1-c_{k,k}(g,h)))\gm_{\A}(g,h)\\
&\leq (1-M_{k,k}M_{k,l}\de_1^2\de_0)\gm_{\A}(g,h).
\end{align*}
We have $M_{k,l}\geq1$ by the  choice of $k,l\in\A$ from Claim~1 or Claim~2 and $M_{k,k}\geq1$. Hence, we infer Claim~3 by letting $\de''$ be the minimum of $M_{k,k}M_{k,l}\de_1^2\de_0$ over all $k,l\in\A$ such that $M_{k,l}\geq1$.

We now prove the statement of the lemma. Let $g\in K$ and $k\in\A$ be taken from Claim~3.
By the primitivity of $M$, we have for all $j\in\A$ the existence of $j({1}),\ldots, j({n})\in\A$ such that $j(1)=j$, $j(n)=k$ and
$M_{j({s}),j({s+1})}\geq1$ for $s=1,\ldots,n$. We compute by iteration, using that $\rho\circ\si_{\zeta}$ is a quasicontraction, (see  Lemma~\ref{l:contraction}~(1.), (2.)), and employing the formula for $\tau$ in Lemma~\ref{l:contraction}~(3.),
\begin{align*}
\gm(\Phi_{j}^{(n)}(g),\Phi_{j}^{(n)}(h))&\leq \sum_{i(1),\ldots,i(n)\in\A,i(1)=j} \ap{\prod_{s=1}^{n}c_{i(s),i(s+1)}p_{i(s),i(s+1)}} \gm(\tau_{i(n)}(g),\tau_{i(n)}(h)).
\end{align*}
Let $J=\{(i(1),\ldots,i(n))\in\A^{n}\mid i(1)=j\}\setminus\{(j(1),\ldots,j(n))\}$.
We factor out $\gm_{\A}(g,h)$ and get, since $c_{i(s),i(s+1)}\leq 1$ and $\gm(\tau_{i(n)}(g),\tau_{i(n)}(h))\leq\gm_{\A}(g,h)$,
\begin{align*}
\ldots&\leq \ab{\sum_{(i(1),\ldots,i(n))\in J} {\prod_{s=1}^{n}p_{i(s),i(s+1)}} +\ap{\prod_{s=1}^{n}p_{j(s),j(s+1)}} \frac{\gm(\tau_{j(n)}(g),\tau_{j(n)}(h))}{\gm_{\A}(g,h)}} \gm_{\A}(g,h).
\end{align*}
As $\sum_{i(1),\ldots,i(n)} \prod_{s=1}^{n}p_{i(s),i(s+1)}=1$ and $j(n)=k$, we get
\begin{align*}
\ldots&\leq\ap{1-\ap{\prod_{s=1}^{n}p_{j(s),j(s+1)}}\ap{1- \frac{\gm(\tau_k(g),\tau_k(h))}{\gm_{\A}(g,h)}}}\gm_{\A}(g,h)\\ &\leq \ap{1-\ap{\prod_{s=1}^{n}p_{j(s),j(s+1)}}
\de''}\gm_{\A}(g,h),
\end{align*}
by using Claim~3 in the second estimate.

By our choice of $j(1),\ldots,j(n)$ the product over the $p_{j(s),j(s+1)}$'s is positive.  We take the minimum over all such positive products to obtain the desired constant $\de>0$.
\end{proof}
\end{lemma}

\textbf{Remark.}
Clearly, for a ball $B$ about some element $h\in\h^{\A}$, there are always $g\in B$ such that $$\ma{\al_{j,k}(g,h)}=0.$$ In particular, this is the case for those $g$ which can be written as $rh$ for some $r>0$. However, we can deal with this problem by considering the part of a ball, where the assumptions of the lemma above fail separately. Details are worked out in the next subsection.

\subsection{Contraction properties of the iterated contraction map}
Recall that $\Sigma_1$ is the set of energies $E\in\R$ for which $\Phi_{E}$ has a fixed point in $\h^{\A}$ and $\Sigma_0$ is that  subset of $\Sigma_1$ for which the components of a fixed point to a given $E$ are linear multiples of each other. Note that for non regular tree operators the set  $\Sigma_0\subseteq\{0\}$ (see Lemma~\ref{l:Sigma0}). In this subsection we prove the following theorem.\medskip

\begin{thm}\label{t:Contraction}(Contraction in $(n+1)$ steps.) For arbitrary $E\in\Sigma_1\setminus\Sigma_0$ and a fixed point $h\in\h^{\A}$ of $\Phi_E$ there are $c\in[0,1)$ and $R>0$  such that for all $g\in B_{R}(h)$
\begin{align*}
\dist_{\h^{\A}}\ap{\Phi_{E}^{n+1}(g),\Phi_{E}^{n+1}(h)}\leq c\;\dist_{\h^{\A}}\ap{g,h},
\end{align*}
where $\Phi_{E}^{n+1}$ means that $\Phi_E$ is applied $(n+1)$ times.
\end{thm}
The strategy of the proof which is given at the end of this subsection is the following: We consider a ball $B$ about a fixed point of $\Phi_{E}$. We decompose $B$ into a part $K$ which satisfies the sufficient criterion for uniform contraction, Lemma~\ref{l:tau}, and show that $\Phi_E$ maps the complement $B\setminus K$ into $K$. Then, again by Lemma~\ref{l:tau} we conclude uniform contraction on this part. Finally to conclude the statement for $d_{\h^{\A}}$ form $\gm_{\A}$ we use Lemma~\ref{l:dist}.

Note that $\tau:\h^{\A}\to\h^{\A}$ extends to a linear map $\C^{\A}\to\C^{\A}$. Moreover, it is easy to check that
$\Phi_{z,j}(g)\neq h_{j}$ implies $\tau_{j}(g-h)\neq 0$ for $j\in\A$.
\medskip

\begin{lemma}\label{l:taual}
Let $z\in\ov\h$ and $h\in\h^{\A}$ be a fixed point of $\Phi_z$. Then, for all $g\in\h^{\A}$, $j,k\in\A$,
\begin{align*}
\al_{j,k}\ap{\Phi_{z}(g),\Phi_{z}(h)} =\arg\ap{\tau_j(g-h)\ov{\tau_k(g-h)}} +\arg\ap{\Phi_{z,j}(g)\ov{\Phi_{z,k}(g)}}
+\arg\ap{h_{j}\ov{h_{k}}},
\end{align*}
where we additionally assume that $\Phi_{z,j}(g)\neq h_{j}$ and  $\Phi_{z,k}(g)\neq h_{k}$.
\begin{proof}
We calculate directly using the decomposition $\Phi_z=\rho\circ\si_z\circ\tau$
\begin{align*}
\al_{j,k}\ap{\Phi_{z}(g),\Phi_{z}(h)} &=\arg\ap{\frac{-1}{\si_{z,j}(\tau(g))}-\frac{-1}{\si_{z,j}(\tau(h))}} \ov{\ap{\frac{-1}{\si_{z,k}(\tau(g))}-\frac{-1}{\si_{z,k}(\tau(h))}}}\\
&=\arg\ap{\frac{\tau_j(g-h)}{\si_{z,j}(\tau(g))\si_{z,j}(\tau(h))}} \ov{\ap{\frac{\tau_k(g-h)}{\si_{z,k}(\tau(g))\si_{z,k}(\tau(h))}}}\\
&=\arg\ap{\tau_j(g-h) \ov{\tau_k(g-h)}}\ap{\Phi_{z,j}(g)\ov{\Phi_{z,k}(g)}} \ap{h_{j}\ov{h_{k}}},
\end{align*}
where we used $\Phi_{z,l}(\cdot)=-1/\si_{z,l}(\tau(\cdot))$, $l\in\A$, and the assumption $\Phi_{z}(h)=h$.
\end{proof}
\end{lemma}
The idea now is to use the formula of the lemma above: If $\ma{\al_{j,k}(g,h)}$ is large for some $j,k$ we can apply the sufficient criterion for uniform contraction, Lemma~\ref{l:tau}, directly. Otherwise, we appeal to Lemma~\ref{l:taual} in the following way:
Suppose $\ma{\al_{j,k}(g,h)}$ is small for all $j,k$. Then, $\ma{\arg(\tau_{j}(g-h)\ov{\tau(g_{k}-h_{k})})}$ is small by a geometric argument. Moreover, if $g$ is very close to $h$, then the last two terms in the formula of the lemma are equal up to a small error. As we know from Lemma~\ref{l:Sigma0}, the last term is non zero except for a finite set of energies in the case of non regular trees. Lemma~\ref{l:taual} then proves that $\ma{\al_{j,k}(\Phi_E(g),h)}$ is large. Therefore,  we can apply the sufficient criterion for uniform contraction, Lemma~\ref{l:tau}, either for $g$ or $\Phi_{z}(g)$.

\medskip

\begin{lemma}\label{l:al}
For all $E\in\Sigma_1\setminus\Sigma_0$ and  a fixed point $h\in\h^{\A}$ of $\Phi_{E}$ there are $R>0$ and $\de>0$ such that for all $g\in B_{R}(h)$ with  $g_{j}\neq h_{j}$ and $\Phi_{z,j}(g)\neq h_{j}$ for all $j\in\A$,
\begin{align*}
\max_{j,k\in\A}\na{\al_{j,k}(g,h)}\geq \de \quad \mbox{ or }\quad \max_{j,k\in\A}\na{\al_{j,k}(\Phi_{E}(g),\Phi_{E}(h))}\geq \de.
\end{align*}
\begin{proof}
By the assumption $E\in \Sigma_1\setminus\Sigma_0$, there are $j,k\in\A$ such that $\de':=\ma{\arg(h_{j}\ov{h_{k}})}>0$. As $h_{j},h_{k}\in\h$, we have $\de'\in(0,\pi)$. We fix $j$, $k$ and $\de'$ for the rest of the proof and set
\begin{align*}
    \de:=\frac{1}{2}\min\set{\de',\pi-\de'}.
\end{align*}
Let $R>0$ be chosen so small, that for all $g\in B_{R}(h)$  and $j\in\A$ we have $\am{\arg(g_{j}\ov h_{j})}< \de$. By the triangle inequality of $\ma{\cdot}$ we get
\begin{align*}
\na{\arg\ap{g_{j}\ov{g_{k}}h_{j}\ov{h_{k}}}}
\geq \na{2\arg\ab{h_{j}\ov{h_{k}}}} -\na{\arg\ab{g_{j}\ov{h_{j}}}}-\na{\arg\ab{\ov{g_{k}}{h_{k}}}}\geq
4\de-2\de=2\de,
\end{align*}
for all $g\in B_{R}(h)$. Since $\Phi_{E}$ is a quasicontraction and $h$ is a fixed point, we have $\Phi_E(B_R(h))\subseteq B_R(h)$. Therefore,  we can conclude by the previous inequality
\begin{align*}
\na{\arg\ap{\Phi_{E,j}(g)\ov{\Phi_{E,k}(g)}} +\arg\ap{h_{j}\ov{h_{k}}}}\geq  2\de.
\end{align*}
Now, since $\tau_{j}$ maps the cone  spanned by the vectors  $g_{k}-h_{k}$, $k\in\A$, into itself we get
\begin{align*}
\max_{j,k\in\A}\na{\arg \ap{\tau_{j}(g-h)\ov{\tau_k(g-h)}}}\leq\max_{j,k\in\A}\na{\arg \ap{(g_{j}-h_{j})\ov{(g_{k}-h_{k})}}}.
\end{align*}
Combining this with Lemma~\ref{l:taual} and the inequality above, we obtain $\na{\al_{j,k}(\Phi_{E}(g),\Phi_{E}(h))}\geq2\de-\na{\al_{j,k}(g,h)}\geq\de$ whenever  $g\in B_{R}(h)$ is such that $\ma{\al_{l,m}(g,h)}<\de$  for all $l,m\in\A$.
\end{proof}
\end{lemma}

We now prove Theorem~\ref{t:Contraction}.

\begin{proof}[Proof of Theorem~\ref{t:Contraction}]
Let $E\in\Sigma_1\setminus\Sigma_0$ and $h\in\h^{\A}$ be a fixed point of $\Phi_{E}$. Let $R>0$ and $\de>0$ be taken from Lemma~\ref{l:al}. We divide the set $B_R(h)$ into two disjoint  subsets
\begin{align*}
B_{\geq}(h)&:=\set{g\in B_{R}(h)\mid g_{j}=h_{j} \mbox{ for some } j\in\A\mbox{ or } \max_{j,k\in\A}\na{\al_{j,k}(g,h)}\geq\de},\\
B_{<}(h)&:=\set{g\in B_{R}(h)\mid g_{j}\neq h_{j} \mbox{ for all } j\in\A\mbox{ and }  \max_{j,k\in\A}\na{\al_{j,k}(g,h)}< \de}.
\end{align*}
We first apply Lemma~\ref{l:tau} with $K=B_{\geq}(h)$: As $\Phi_{E}$ is a quasicontraction, we obtain for $g\in B_{\geq}(h)$
\begin{align*}
\gm_{\A}\ap{\Phi_E^{n+1}(g),\Phi_E^{n+1}(h)}\leq \gm_{\A}\ap{\Phi_E^{n}(g),\Phi_E^{n}(h)} \leq (1-\de')\gm_{\A}(g,h),
\end{align*}
with some $\de'>0$ which is independent of $g$.
For $g\in B_{<}(h)$ we have by Lemma~\ref{l:al}, as $\Phi_{E}$ is a quasi contraction and $h$ is a fixed point,
\begin{align*}
    \Phi_{E}(B_{<}(h))\subseteq B_{\geq}(h).
\end{align*}
Therefore,   Lemma~\ref{l:tau} with $K= \Phi_{E}(B_{<}(h))$ applied to $\Phi_E(g)$ for $g\in B_{<}(h)$ yields
\begin{eqnarray*}
\gm_{\A}\ap{\Phi_E^{n+1}(g),\Phi_E^{n+1}(h)} &=&  \gm_{\A}\ap{\Phi_E^{n}(\Phi_{E}(g)),\Phi_E^{n}(\Phi_{E}(h))}\\
 & \leq &  (1-\de')\gm_{\A}(\Phi_{E}(g),\Phi_{E}(h))\\
 & \leq & (1-\de')\gm_{\A}(g,h).
\end{eqnarray*}
By Lemma~\ref{l:dist},
we get the existence of $c\in[0,1)$ such that
$$\dist_{\h^{\A}}\ap{\Phi_E^{n+1}(g),\Phi_E^{n+1}(h)} \leq c\, \dist_{\h^{\A}}(g,h),$$
for $g\in B_R(h)$ since $\gm_{\A}\ap{\Phi_E^{n+1}(g),\Phi_E^{n+1}(h)} \leq (1-\de')\gm_{\A}(g,h)$ for $g\in B_R(h)$.
\end{proof}

\subsection{Continuity and stability of fixed points}
In this subsection we will use Theorem~\ref{t:Contraction} proven above to show that fixed points depend continuously on the energy and the potential.

\begin{thm}\label{t:continuity}(Continuity, uniqueness and stability of fixed points.) Let  $E\in\Sigma_1\setminus\Sigma_0$ be given and $h\in \h^{\A}$ be a  fixed point of $\Phi_E$. Then there is $r>0$ such that for all $z\in\ov\h$ with $|z-E|\leq r$ the map $\Phi_{z}$ has a unique fixed point $h(z)$ which depends  continuously on   $z$.
In particular,  the set $\Sigma_1\setminus\Sigma_0$ is open in $\R$. Furthermore, there exists $R>0$ such that for all $\xi\in\ov\h^{\A\times\N_0}$, with components $\xi_j(k)$, $j\in\A$, $k\in\N_0$ satisfying $|\xi_j(k)-E|\leq r$, the inclusion
\begin{align*}
\Phi_\zeta^{(m)}(B_R(h))\subseteq B_R(h)
\end{align*}
holds for all $m\in\N$, where  $\Phi_{\zeta}^{(m)}:=\Phi_{\zeta(1)}\circ \cdot \circ \Phi_{\zeta(m)}$.
\begin{proof}Let $U_r(E):=\{(\zeta(1),\ldots,\zeta(n+1))\in\ov\h^{\A}\mid |\zeta_j(m)-E|\leq r $ for all $j\in\A,m=1,\ldots,n+1\}$, where $n=n(M)$ is the primitivity exponent of $M$.
Define the function
\begin{align*}
d:[0,\infty)\times[0,\infty)\to[0,\infty),\quad
(r,R)\mapsto\max_{\zeta\in U_r(E)}\max_{g\in B_R(h)}\dist_{\h^{\A}}(\Phi^{(n+1)}_{\zeta}(g),h),
\end{align*}
which is continuous and satisfies $d(r,R)\to0$ for $r,R\to 0$ by continuity of $\Phi_{\zeta}^{(n+1)}(g)$ in $g$ and $\zeta$ and continuity of $\dist_{\h^{\A}}$.
Let $c$ be the contraction coefficient and $R>0$ be taken from Theorem~\ref{t:Contraction}.
For arbitrary small $\eps\in(0,R)$, let now $r>0$ be such that $d(r,\eps)\leq (1-c)\eps$. Then, for all $\zeta\in U_{r}(E)$ with and $g\in B_{\eps}(h)$ we have by Theorem~\ref{t:Contraction} as $h$ is a fixed point of $\Phi_{E}$
\begin{align*}
\dist_{\h^{\A}}(\Phi^{(n+1)}_{\zeta}(g),h)&\leq \dist_{\h^{\A}}(\Phi^{(n+1)}_{\zeta}(g),\Phi^{n+1}_{E}(h) ) +
\dist_{\h^{\A}}(\Phi^{n+1}_{E}(g),\Phi^{n+1}_{E}(h))\\
&\leq d(r,\eps)+c\,\dist_{\h^{\A}}(g,h)\leq (1-\eps)c+c\eps=\eps.
\end{align*}
We conclude $\Phi^{(n+1)}_{\zeta}(g)\in B_{\eps}(E)$.\\ Hence, the last statement follows as any $m\in\N$ can be decomposed into $m=kn+l$, $k,l\in\N_0$, $l\leq n$ and $\Phi_{\zeta}^{(l)}$ maps $B_{\eps}(h)$ into a compact ball.\\
Let us turn to the first statements. We now consider $\zeta(1)=\ldots=\zeta(n+1)=(z,\ldots,z)$. From the considerations above we conclude that any accumulation point of $(\Phi_{z}^{m}(g))_{m\in\N}$ lies in $B_\eps(h)$ for $g\in B_{\eps}(h)$ and $z$ sufficiently close to $E$. For $z\in\h$, we know by Proposition~\ref{p:FHS} that these accumulation points must actually be fixed points of $\Phi_{z}$. Moreover, they are unique. Therefore, it easily follows that $h$ is the unique fixed point of $\Phi_{E}$. As this uniqueness holds for any $E\in\Sigma_1\setminus \Sigma_0$ the statement follows in particular for all $z\in\ov\h$ with $|z-E|\leq r$ for some fixed $E$.
\end{proof}
\end{thm}

\section{Proof of the theorems}
In this section we prove Theorem~\ref{main1}, Theorem~\ref{main2} and Theorem~\ref{main3}. In \eqref{e:Gmdef} we defined the vector  $\Gm(z)\in \h^{\A}$ via the truncated Green functions $\Gm_{a(x)}:=\Gm_{x}(z,\Lp)$, $x\in V$. In the first subsection we study the Green function in the limit $\Im z\downarrow 0$ and then we turn to the proofs of the theorems in the following subsection.

\subsection{The Green function in the limit} Define $\mathcal{U}$ to be the system of all  open sets $U\subset \R$ such that for each $x\in V$
 the function $\h\longrightarrow \h,\; z\mapsto \Gm_{x}(z,\Lp)$ can uniquely be extended to a  continuous
 function from  $\h\cup U$ to $ \h$ and set

\begin{align*}
\Sigma:=\bigcup_{U\in \mathcal{U}} U.
\end{align*}
This entails that for $E\in\Sigma$  and $x\in  V$ the limits  $\Gm_{x}(E,\Lp)=\lim_{\eta\downarrow0}\Gm_{x}(E+i\eta,\Lp)$
exist, are continuous in $E$ and $\Im \Gm_{x}(E,\Lp)>0$.

The following theorem directly implies Theorem~\ref{main1}.

\begin{thm}\label{t:Sigma}
The set $\Sigma$ consists of finitely many open intervals and
\begin{align*}
    \si(\Lp)=\clos \Sigma.
\end{align*}
Moreover, for every $x\in V$ the Green function
$\h\to\h$, $z\mapsto G_{x}(z,\Lp)$,
has a unique continuous extension to $\h\cup\Sigma \to\h$ and is uniformly bounded in $z$.
\end{thm}

For a regular tree with branching number $k$ the theorem above is well known. In this case,  the vector $\Gm(z)$ consists of only one component which can be explicitly calculated from the recursion formula \eqref{e:Gm2}. In the limit $\Im z\downarrow 0$, i.e., for $z\to E$, one gets $\Gm(E)=(-E+\sqrt{E^2-4k})/2k$. As this is well known, we will restrict our attention for the rest of the section to the case of non regular trees. In particular, in this case the set $\Sigma_0\subseteq\{0\}$, see Lemma~\ref{l:Sigma0} in Subsection~\ref{ss:Sigma0}. (As for regular trees $\Sigma_0=\Sigma_1$ and thus $\Sigma_1\setminus\Sigma_0=\emptyset$ all the following statements are true, but, however, pointless in this case.)

\begin{lemma}\label{l:Gmlimits}  Let $E\in\R\setminus\Sigma_0$. Then, the following holds:  \\
(1.) The limit $\Im    \Gm(E):=\lim_{\eta\downarrow 0}\Im \Gm(E+i\eta)$ exists in $[0,\infty)^{\A}$  and it is continuous in $E$ on $\R\setminus \Sigma_0$. \\
(2.)  $E\in\Sigma_1$ is equivalent to $\Im \Gm_j(E)>0$ for some $j\in\A$. In this case, the limit
\begin{align*}
\Gm(E):=\lim_{\eta\downarrow 0}\Gm(E+i\eta)
\end{align*}
exists in $\h^{\A}$ and it is continuous in a neighborhood of $E$.
\begin{proof}
As $\Gm(z)$ satisfies the recursion relation \eqref{e:Gm2}, it is a fixed point of $\Phi_{z}$ for $z\in\h$. Moreover, by Proposition~\ref{p:FHS} it is the unique fixed point. By Lemma~\ref{l:accumulationpoints} the accumulation points of these fixed points as $\Im z\downarrow0$ are again fixed points of $\Phi_{E}=\Phi_{\Re z}$ and lie either in $\R^{\A}$ or in $\h^{\A}$.\\
Let $E\in\R\setminus \Sigma_0$. Suppose first there is an accumulation point $h$ of $\Gm(z)$ as $z\to E$  which lies in $\h^{\A}$. Then, by uniqueness and continuity of fixed points, Theorem~\ref{t:continuity}, this $h$ is the unique fixed point of $\Phi_E$ and we denote  $\Gm(E):=h$. The continuity follows also by Theorem~\ref{t:continuity}.\\
If, on the other hand,  $h\in\R^{\A}$ for all accumulation points we conclude that $\Im \Gm(z)\to0$ as $z\to E$. Hence, we know that the limit of the imaginary part exists.\\
This gives the assertions (1.) and (2.).
\end{proof}
\end{lemma}

\begin{lemma}\label{l:Sigma} The set $\Sigma$ consists of finitely many open intervals. Moreover,
$\Sigma_1\setminus\Sigma_0\subseteq\Sigma\subseteq\Sigma_1.$
\begin{proof}
We first check the inclusions.
The first inclusion follows from Theorem~\ref{t:continuity}.
The second inclusion is due to the fact that the truncated Green functions solve the polynomial equations.
We know that $\Sigma_1$ has finitely many connected components by Lemma~\ref{l:milnor} and  the set $\Sigma_0$ is finite by Lemma~\ref{l:Sigma0}. Hence, $\Sigma$ consists of finitely many intervals and it is open by definition.
\end{proof}\end{lemma}

We are now prepared to prove Theorem~\ref{t:Sigma}.

\begin{proof}[Proof of Theorem~\ref{t:Sigma}] The case of regular trees was discussed right after the statement of the theorem. So, we consider only the case of non regular trees.\\
By the previous lemma, the set $\Sigma$ consists of finitely many intervals. As $\Gm_x$ is continuous by definition of $\Sigma$ we have that the maps $G_{x}:\h\to\h$ have continuous extensions to $\h\cup\Sigma\to\h$ for $x\in V$ by Proposition~\ref{p:G} and  Proposition~\ref{p:G2} (3.). As the measure $\Gm_x(E+i\eta,\Lp)dE$ converges  vaguely to the spectral measure $\mu_x$ as $\eta\downarrow0$, this yields  $\si(T)\supseteq\clos (\Sigma)$.\\
By Lemma~\ref{l:Gmlimits} we have that $\Im\Gm_{j}(E)=0$ for  $E\in\R\setminus\Sigma_1$ and $j\in\A$. Moreover,   Lemma~\ref{l:Gmbounds} gives the uniform boundedness of $\Gm_j(z)$ in $z$. Thus, by Proposition~\ref{p:G2}~(2.), we conclude that $\Im G_{x}(E)=0$ for all $E\in\R\setminus\Sigma_1$ and $x\in V$. Hence, $\si(T)\subseteq\clos(\Sigma_1)$.
As $\Sigma_0\subseteq\{0\}$  by Lemma~\ref{l:Sigma0}, the sets $\clos(\Sigma_1\setminus\Sigma_0)$ and $\clos(\Sigma_1)$ can differ by only by an isolated point which can only support a point measure. However, the uniform boundedness of $\Gm_{x}(z,T)$, which extends to $G_{x}(z,T)$, makes sure that eigenvalues do not occur. Therefore, we get $\si(T)=\clos(\Sigma_1\setminus\Sigma_0)$. By the previous lemma we obtain $\si(T)=\clos(\Sigma_1\setminus\Sigma_0)\subseteq\clos(\Sigma)$ which finishes the proof.
\end{proof}

\subsection{Absolutely continuous spectrum for the free operator and stability under radially label symmetric potentials}

\begin{proof}[Proof of Theorem~\ref{main1}]
This follows from Theorem~\ref{t:Sigma} as the spectral measures $\mu_x$ satisfy $\mu_x=\Im G_x(E)dE$ and hence  are absolutely continuous and supported on finitely many intervals.
\end{proof}

We now turn to the proof of  Theorem~\ref{main2}.

\begin{proof}[Proof of Theorem~\ref{main2}] Let $z\in\h$ and $v$ a radially label symmetric potential be given. Define $w:\N_0\times\A\to[-1,1]$ as $w_{a(x)}(|x|)=v(x)$ and $w_{j}(n)=0$ for $n\in\N_0$, $j\in\A$ where there is no $x\in V$ such that $(n,j)=(|x|,a(x))$. Moreover, let $\zeta_{j}(n)=-w_j(n)+z$ for $n\in\N_0$, $j\in\A$.
By the symmetry of the potential, the truncated Green function $\Gm_x(z,\Lp+\lm v)$ depends only on $|x|$ and $a(x)$. By Proposition~\ref{p:FHS} we have that $$\Gm_{x}(z,\Lp+\lm v)=\lim_{n\to\infty}\Phi_{\zeta(|x|),a(x)} \circ\ldots\circ\Phi_{\zeta(|x|+n)}(g)$$
for all $g\in\h^{\A}$. Let $E\in \Sigma_1\setminus\Sigma_0$. Then, by Theorem~\ref{t:continuity} there exists $\lm(E)>0$ such that that $\Gm_{x}(z,\Lp+\lm(E) v)$ lies in a ball about $\Gm_{x}(E,\Lp)$ for $z\in\h$ close to $E$ and all  radially label symmetric potentials $v$. For a closed set $I$ which is included in the interior of $\sigma(\Lp)\setminus\{0\}\subseteq\sigma(\Lp)\setminus\Sigma_0$ there is $\lm>0$ that $\Gm (z,\Lp+\lm v)$ is uniformly bounded in $z$ for $\Re z\in I$.\\
Thus, the spectral measures $d\mu_{x}=\Im \Gm (E,\Lp+\lm v)dE$, $x\in V$, are absolutely continuous.
\end{proof}

Theorem~\ref{main3} is  a direct consequence of Theorem~\ref{main2}.

\begin{proof}[Proof of Theorem~\ref{main3}] Multiplication by a  bounded $v$ which is vanishing at infinity is a compact operator. Therefore, we have preservation of the essential spectrum $\si_{\mathrm{ess}}(\Lp+v)=\si_{\mathrm{ess}}(\Lp)$. Moreover, by Theorem~\ref{main1}, the spectrum of $\Lp$ is purely absolutely continuous.
Thus, $\si_{\mathrm{ac}}(\Lp+ v)\subseteq \si_{\mathrm{ess}}(\Lp+ v)=\si_{\mathrm{ess}}(\Lp) \subseteq \si(\Lp)=\si_{\mathrm{ac}}(\Lp)$.\\
Conversely, the absolutely continuous spectrum is stable under finitely supported perturbations. In particular, setting $v$ zero at the vertices $x$ where $|v(x)|\geq\lm$ leaves the absolutely continuous spectrum of $\Lp+v$ invariant. By Theorem~\ref{main2}  the absolutely continuous spectrum of every compact subset included in the interior of  $\si(\Lp)\setminus\{0\}$ can be preserved under perturbations by sufficiently small radially label symmetric $v$. Since $v$ vanishes at infinity, every such interval is contained in the absolutely continuous spectrum of $\Lp+v$.
\end{proof}

\textbf{Acknowledgements.} The research was financially supported by a grant (MK) from Klaus Murmann Fellowship Programme (sdw), in part by a Sloan fellowship (SW) and NSF grant DMS-0701181 (SW). Part of this work was done while MK was visiting Princeton University. He would like to thank the Department of Mathematics for its hospitality.



\end{document}